\documentclass{article}
\usepackage[utf8]{inputenx}
\usepackage{tikz-cd}
\usepackage{indentfirst}
\usepackage{microtype}
\usepackage{amsfonts}
\usepackage{mathtools}
\usepackage[english]{babel}
\usepackage{euscript}
\usepackage{setspace}
\usepackage{amsthm}
\usepackage[bottom]{footmisc}
\usepackage{bbm}
\usepackage{hyperref}
\usepackage{xpatch}
\usepackage{pgfplots}
\pgfplotsset{compat = 1.15}
\usetikzlibrary{arrows.meta}
\usepackage{amsmath,url}
\usepackage{enumitem}
\usepackage{wrapfig}
\usepackage{caption}
\usepackage{float}
\usepackage{bm}
\usepackage [a4paper, top=3.5cm, bottom=3.5cm, left=3.3cm, right=3.3cm] {geometry}
\usepackage{comment}
\usepackage{todonotes}
\usepackage{amssymb}

\usepackage{mathrsfs} 

\usepackage{resmes}

\usepackage{graphicx}
\usepackage[export]{adjustbox}
\graphicspath{ {./images/} }

\usepackage{lettrine}
\theoremstyle{definition}
\newtheorem{definition}{Definition}[section]
\newtheorem{theorem}[definition]{Theorem}
\newtheorem{proposition}[definition]{Proposition}
\newtheorem{example}[definition]{Example} 
\newtheorem{rk}{Remark}[section]
\newtheorem{lemma}[definition]{Lemma}

\newtheorem{hyp}[definition]{Hypothesis}

\let\oldexample\example
\renewenvironment{example}
  {\oldexample\pushQED{\qed}\ignorespaces} 
  {\hfill$\blacklozenge$\par}                  

\usepackage{stackengine}

\newcommand{\jump}{\hfill\break}
\newcommand{\R}{\mathbb{R}}

\newcommand{\C}{\mathbb{C}}
\newcommand{\Z}{\mathbb{Z}}

\newcommand{\e}{\textnormal{e}}
\renewcommand{\u}{\textnormal{u}}
\renewcommand{\i}{\textnormal{i}}

\title{The Mean-Field Ott-Antonsen Manifold is an Unstable Manifold in the Continuum Limit}
\author{Christian Kuehn, Giacomo Landi}
\date{\today}

\begin{document}

\maketitle

\begin{abstract}
We study interacting particle systems of Kuramoto-type. Our focus is on the dynamical relation between the partial differential equation (PDE) arising in the continuum limit (CL) and the one obtained in the mean-field limit (MFL). Both equations arise when we are considering the limit of infinitely many interacting particles but the classes of PDEs are structurally different. The CL tracks particles effectively pointwise, while the MFL is an evolution for a typical particle. First, we briefly discuss the relation between solutions of the CL and the MFL showing how to generate solutions of the CL starting from solutions of the MFL. Our main result concerns a dynamical relation between important invariant manifolds of the CFL and the MFL. In particular, we give an explicit proof that the unstable manifold of the homogeneous steady state of the CL is the direct dynamical analogue of the famous Ott-Antonsen manifold for the MFL.
\end{abstract}
\section{Introduction}
Interacting particle systems form a broad class of coupled differential equations used to describe populations of interacting agents (“particles”) that can display collective behavior without any centralized control. Such systems are employed across a wide range of scientific disciplines to model diverse phenomena, including opinion dynamics \cite{hegselmann2015opinion}, collective organization in animal groups \cite{青木一郎1982simulation}, \cite{ballerini2008interaction}, \cite{lopez2012behavioural}, synchronization of coupled oscillators \cite{kuramoto2005self}, and biological tissue formation \cite{barre2019modelling}. These models have also attracted considerable attention in the mathematical community, especially in the context of Kuramoto-type systems \cite{acebron2005kuramoto,gupta2014kuramoto}, with particular focus on their behavior in the large-population (infinite-particle) limit, which is the main subject of this work.

The relevance of studying large-population limits becomes clear when considering the size of certain systems of interest. For example, starling flocks can contain several thousand individuals \cite{ballerini2008interaction}, while mammalian neural networks may involve up to $10^{11}$ neurons \cite{jabin2025mean}. Simulating such massive populations by directly integrating a system of $N$ coupled equations, with potentially $N^2$ pairwise interactions, quickly becomes computationally prohibitive as $N$ grows. To overcome this challenge, a key idea, introduced in the seminal works \cite{braun1977vlasov}, \cite{dobrushin1979vlasov}, \cite{golse2016dynamics}, \cite{jabin2018quantitative}, \cite{neunzert2006approximation}, \cite{serfaty2020mean}, \cite{sznitman2006topics}, is to shift the focus from the individual particle trajectories to the evolution of the population’s density, leading to a mean-field description of the system.\jump\jump
\textbf{\textit{Mean-field limit of exchangeable particle systems.}}
For each particle, we fix a phase space $\mathcal{X}\subseteq \mathbb{R}^d$ with $d\in\mathbb{N}$. The configuration of $N$ interacting particles is denoted by $(x_i^N)_{i\in\{1,\cdots,N\}}$ and evolves according to this system of coupled ordinary differential equations:
\begin{equation}\label{DS}
    \begin{cases}\tag{DS}
    \dot{x}^N_i(t)=\frac{1}{N}\sum_{j=1}^N \phi(x_i^N(t),x_j^N(t))\qquad\forall i=1,\ldots,N,\quad t\in[0,T]\\
    x^N_i(0)=x^{N,0}_i\qquad\forall i=1,\ldots,N,
    \end{cases}
\end{equation}

where
\begin{itemize}
    \item the index set $\{1,\ldots,N\}$ labels the individual particles, characterizing their identities;
    \item $(x_i^N)_{i\in\{1,\cdots,N\}}$ represents the state of the interacting particles, and may represent positions \cite{faure2015crowd}, opinions \cite{hegselmann2015opinion}, velocities \cite{cucker2007emergent}, phases \cite{kuramoto2005self}, or angles \cite{vicsek1995novel}. The model then determines the dimension $d$ of the state space;
    \item $\phi:\R^d\times\R^d\rightarrow\R^d$ encodes the pairwise interaction between particles. We will refer to it as the \textit{interaction function} of the system.
\end{itemize}
Under standard assumptions on $\phi$, the solution $(x_i^N)_{i\in\{1,\cdots,N\}}$ of the system exists in $C([0,T];(\R^d)^N)$.
 As $N\rightarrow+\infty$ one can show that the system \eqref{DS} is well approximated by its so-called mean-field limit $\mu \in C([0,T]; \mathcal{P}(\mathbb{R}^d))$, which satisfies the following equation:
 \begin{equation}\label{MFL}
\begin{cases}\tag{MFL}
    \partial_t \mu_t(x)+\nabla_x\cdot\left(\mu_t(x)\int_{\R^d}\phi(x,y)\text{d}\mu_t(y)\right)=0,\\
    \mu_{t=0}=\mu_0.
\end{cases}
\end{equation}
commonly referred to as the Vlasov equation in the context of interacting particle systems (IPS). When the interaction function $\phi$ is a gradient, the equation is known in the literature as the aggregation equation. The interpretation of $\mu$ is usually the following one: let $\Delta x\subset\R^d$, then $\mu_t(\Delta x)$ represents the probability of finding an agent with position in $\Delta x$ at time $t$.
 The relationship between the particle system \eqref{DS} and its mean-field limit equation \eqref{MFL} is elucidated by introducing the empirical measure $$\mu_t^N\coloneqq\frac{1}{N}\sum_{i=1}^N\delta_{x_i^N(t)}$$  which can be shown to satisfy the mean-field equation \eqref{MFL}, see \cite[Th. 1.3.1]{golse2016dynamics}. The convergence of the microscopic dynamics $(x_i^N)_{i=1}^N$ governed by \eqref{DS} towards the mean-field solution $\mu$ of \eqref{MFL} can be derived using a stability argument for the mean-field equation; see \cite{golse2016dynamics, dobrushin1979vlasov}. This argument yields an estimate of the form: $$W_1(\mu_t,\mu_t^N)\leq C(T)W_1(\mu_0,\mu_0^N),$$ where $W_1$ is the 1-Wasserstein distance \cite{dobrushin1979vlasov}. See also \cite{loeper2006uniqueness} where Wasserstein distances with different exponents different from $1$ are used in the same context. For a detailed introduction on Wasserstein distances, we refer the reader to \cite{villani2008optimal}.
\jump\jump
\textbf{\textit{The continuum limit.}}
The previously discussed mean-field limit establishes the weak convergence of the microscopic dynamics \eqref{DS} towards the solution of the transport equation \eqref{MFL}.
An alternative approach, known as the continuum limit, instead yields pointwise convergence of the solutions of \eqref{DS} to those of an integro-differential equation:
\begin{equation}\label{CL}
\begin{cases}\tag{CL}
    \partial_t x(t,\xi)=\int_0^1 \phi\big(x(t,\xi),x(t,z)\big)\textnormal{d}z,\\
    x(0,\xi)=x_0(\xi)
\end{cases}
\end{equation}
where, in this case, the solution $x(t,\xi)$ represents the position in $\R^d$ at time of the particle labeled with $\xi$, where $\xi\in[0,1]$ is a continuous representation of particle's label. This limit is as important as the previous one, especially in the context of network dynamics and graph limits, two research fields that have seen significant development in recent years, see for example \cite{medvedev2014nonlinear,medvedev2018continuum,kuehn2019power,gkogkas2021continuum}.

\begin{rk}\label{Kura}
The general case of non-exchangeable particles, i.e., when the interaction function is of the form $\phi=\phi(\xi,\zeta,x,y)$ is also of big interest, see for example \cite{kuehn2022vlasov}, \cite{medvedev2014nonlinear} and \cite{ayi2024large}. One important and remarkable example is the Kuramoto model
$$\dot{\theta}_i=\omega_i+\frac{K}{N}\sum_{j=1}^N \sin(\theta_j-\theta_i),\qquad\forall i=1,\ldots,N$$
where the $\omega_i$ represent the natural frequencies of the oscillators, $K$ the coupling strength between the oscillators and $\theta_i\in \mathbb{S}^1$ is the phase of oscillator $i$, see \cite{acebron2005kuramoto}, \cite{rodrigues2016kuramoto} and \cite{gupta2014kuramoto} for a broader and complete perspective about this model. Here, the non-exchangeability is encoded in the different natural frequencies $\omega_i$.
\end{rk}

Understanding the relations and the properties of the \eqref{CL} and the \eqref{MFL} is as important and challenging as there are no previous results available that directly link important dynamically invariant sets in both PDE limits of IPS. The goal of this paper is to deepen the connection between the \ref{MFL} and the \ref{CL} PDEs. In particular, we will show how to obtain solutions of \ref{CL} starting from solutions of \ref{MFL}. The second and main result of the paper is the explicit form of a family of unstable manifolds for the \ref{CL} obtained applying the previous transformation to the Ott-Antonsen (OA) manifold of the \ref{MFL}.\\

The OA manifold has been discovered in the seminal paper \cite{ott2008low} and since then has been widely studied and discussed, see for example \cite{tyulkina2018dynamics}, \cite{ott2009long}, \cite{engelbrecht2020ott} and \cite{martens2009exact}. It provides a remarkable analytic reduction of the mean-field limit of large ensembles of coupled phase oscillators with a particular choice of the interaction function.

The OA manifold can be utilized in a broad class of systems and we briefly review the intuition behind the OA manifold construction here. Suppose we want to determine a mean-field density $f(\theta,\omega,t)$, where $\theta$ is the phase of an oscillator and $\omega$ is the natural frequency an oscillator would have in isolation from the outside world; $f\textnormal{d}\theta \textnormal{d}\omega$ is the fraction of oscillators at time
$t$ whose phases and natural frequencies lie in the infinitesimal ranges  $\textnormal{d}\theta $ and $\textnormal{d}\omega$. Because the natural frequency $\omega$ of each oscillator does not vary with time, the marginal frequency distribution is time independent, i.e.
$$\int_0^{2\pi}f(\theta,\omega,t)\textnormal{d}\theta =g(\omega).$$
The key quantity that captures the macroscopic behavior of the distribution function $f$ is the order parameter $r(t)$, originally introduced in \cite{kuramoto2003chemical} and defined by
$$r(t)=\int_\R\int_0^{2\pi} f(\theta,\omega,t)\textnormal{e}^{-\i\theta}\textnormal{d}\theta \textnormal{d}\omega.$$
The general system in which the OA manifold appears is the one governed by the Vlasov-type equation
\begin{equation}\label{Continuity}
    \partial_tf+\partial_\theta\left(fV\right)=0
\end{equation}
where $V(\theta,\omega,t)$ can be expressed as
$$V(\theta,\omega,t)=\omega+\frac{1}{2\i}\left[H(t)\e^{-\i\theta}-H^*(t)\e^{\i \theta}\right]$$
where $H^*$ indicates the complex conjugate of $H$. Different choices of $H$ could be the following:  for the classical Kuramoto \cite{kuramoto2003chemical} problem, $H=Kr(t)$, where $K$ is the strength of the coupling between oscillators, i.e.,
\begin{equation}\label{MFL-Kura}
    \partial_t f(\theta,\omega,t)+\partial_\theta\left[f\left(\omega+K\int_0^{2\pi}\int_\R f(\phi,\omega',t)\sin(\phi-\theta)\textnormal{d}\omega'\textnormal{d}\phi\right)\right]=0.
\end{equation}
Moreover, in the context of the circadian rhythm model studied in \cite{ott2008low,childs2008stability}, one has $H = K r(t) + \Lambda$, where $\Lambda$ is a constant representing the strength of the diurnal forcing associated with the day–night light cycle. The term $\Lambda$ may also be made time-dependent, $\Lambda = \Lambda(t)$, to account for variations between sunny and cloudy days.

In models incorporating time delays in the response of oscillators to one another \cite{yeung1999time,lee2009large}, the interaction term takes the form $H = K \int_0^\infty \rho(\tau), r(t-\tau), \mathrm{d}\tau$, where $\rho(\tau)$ denotes the distribution of delays along the coupling links \cite{lee2009large}.

For the case of pedestrian-induced oscillations of footbridges \cite{abdulrehem2009low,strogatz2005crowd}, the interaction is given by $H = K \ddot{y}(t)$, where $\ddot{y}(t)$ is the lateral acceleration of the bridge. The bridge motion itself follows a damped oscillator equation driven by $r(t)$, which represents the collective effect of the pedestrians.\\

Then expanding $f$ in Fourier series in $\theta$, we have
$$f(\theta,\omega,t)=\frac{g(\omega)}{2\pi}\left\{1+\sum_{n=1}^\infty f_n(\omega,t)\e^{\i n\theta}+f_n^*(\omega,t)\e^{-\i n\theta}\right\}$$
the Ott-Antonsen ansatz consists in considering a restricted class of Fourier coefficients $f_n(\omega,t)$ such that
$$f_n(\omega,t)=(\alpha(\omega,t))^n$$
substituting this series expansion in \ref{Continuity} one obtains a Riccati equation for $\alpha$
$$\partial_t\alpha+\i\omega\alpha+\frac{1}{2}(H\alpha^2-H^*)=0.$$
Since it can be shown that the time evolution of $\alpha$ stays inside the unit disk, one can perform the summation and obtain
$$f(\omega,\theta,t)=\frac{g(\omega)}{2\pi}\frac{1-\rho^2}{(1-\rho)^2+4\rho\sin^2(\frac{1}{2}(\theta-\psi))}$$
where $\alpha\equiv\rho e^{-\i\psi}$.\\

One approach to try to justify the formal ansatz was to observe that the incoherent state $\rho=0$ and the partially synchronized state $\rho=\textnormal{const.}>0$ satisfy it. Moreover, certain solutions obtained by the OA ansatz can be linked to connecting orbits from the incoherent state to a (partially) synchronized one. Remarkably, for the classical Kuramoto $H=Kr(t)$, on the OA manifold the system can be reduced to this finite set of ODEs
$$
\begin{cases}
    \partial_t \rho+(1-\frac{K}{2})\rho+\frac{K}{2}\rho^3=0,\\
    \partial_t \psi =0
\end{cases}
$$
if we consider $g(\omega)$ to be a Lorentzian distribution. Thus the dynamics is described by a single real nonlinear, first order, ordinary differentia equation.  Moreover it has been shown in \cite{ott2009long} that, unless $g(\omega)$ is singular, i.e. a Dirac's delta, which is the case of homogeneous oscillators with same natural frequencies, the Ott-Antonsen manifold is attracting, so that describes the long behavior of every solution of such mean-field limit equation.

The precise link of the OA ansatz with the finite dimensional system and the Watanabe-Strogatz theory \cite{watanabe1993integrability},\cite{watanabe1994constants} for identical oscillators has been explored in \cite{vlasov2016dynamics},\cite{marvel2009identical},\cite{cestnik2022exact}. Other references regarding the properties of the Ott-Antonsen manifold, like the existence of periodic orbits and discussions on its attractiveness, can be found in \cite{omel2022periodic},\cite{pietras2016ott},\cite{engelbrecht2020ott}. In this paper, we are going to provide a completely different view on the OA manifold, namely that it can be directly matched to an important unstable manifold in the CL PDE. 

\section{Main results}
In this section we will state our main results and prove them. For the sake of clarity we would like to point out that in the next subsection we shall present the same result in two different settings. The first one deals with unbounded domains, which are of interest as a general framework commonly considered in the literature, while the second one concerns periodic domains, as we are particularly interested in the continuum limit of the Kuramoto model. We will point out when the switch will happen between the two cases. The first case covered will be the phase space $\mathbb{R}$, while the second case will be the one-dimensional periodic phase space case, i.e. when working on the circle $\mathbb{S}^1=\mathbb{R}/2\mathbb{\pi Z}$. 

\jump\jump
\textbf{\textit{Relationship between solutions of continuum and aggregation equation.}}
Let $I=[0,1]$ be the unit interval, which tracks the particle label in the continuum limit \eqref{CL}. It has been shown in \cite{biccari2019dynamics} that given a solution $x(t,\xi)\in C([0,T];L^\infty(I,\R^d))$ of \eqref{CL} it is possible to obtain a solution $\Bar{\mu}_t(x)$ of \eqref{MFL} using the "continuous" empirical measure
\begin{equation}\label{transf2}
    \Bar{\mu}_t(x)=\int_I\delta_{x(t,\xi)}\textnormal{d}\xi.
\end{equation}
Now we would like to show that, under certain hypothesis, it is possible to construct a solution of the continuum limit \eqref{CL} starting from one of the aggregation equation \eqref{MFL}. This idea has been inspired from a proposition in \cite{ayi2024large} that states that for any $\mu\in\mathcal{P}(\R)$, there exists a measurable function $x:I\rightarrow\R$ such that $\mu(x)=\int_I\delta_{x(\xi)}(x)\textnormal{d}\xi$ and such $x(\xi)$ is given by 
\begin{equation}\label{transf}
    x(\xi)\coloneqq\inf\left\{x\in\R\;\bigg\vert\; F_\mu(x)\geq\xi\right\}
\end{equation}
where $F_\mu(x)$ is the cumulative distribution function (CDF) of $\mu$, i.e
$$F_\mu(x)=\int_{-\infty}^x \textnormal{d}\mu(y).$$
From now on we will only work with one-dimensional phase spaces, due to the properties of the CDF and the just mentioned proposition of \cite{ayi2024large}. In the following we will assume that the probability measure $\mu_t$ solving \eqref{MFL} admits a density $f_t$, whose regularity will be discussed right after Hypothesis 2.1, that satisfies the \eqref{MFL}, i.e., 
\begin{equation}
\begin{cases}
    \partial_t f_t(x)+\partial_x\left(f_t(x)\int_{\R}\phi(x,y)f_t(y)\textnormal{d}y\right)=0,\\
    f_{t}(x)|_{t=0}=f_0(x),
\end{cases}
 \end{equation}
 where $f_0$ is a given initial density. Moreover we suppose that $f_t(x)>0$ for a.e. $(t,x)\in[0,T]\times\R$. We make this assumption in order to have that $F_{\mu}(x)$ is invertible and strictly monotone so that the pseudo-inverse corresponds to the actual inverse
 $$\inf\left\{x\in\R\;\bigg\vert\; F_\mu(x)\geq\xi\right\}=F^{-1}_\mu(\xi).$$
Now we will state the hypothesis on the initial datum and the interaction function that we will use in order to have sufficient regularity of the solutions.
\begin{hyp}\hfill
    \begin{itemize}
    \item[(A1)] We suppose that the initial datum $f_0$ is such that 
    $$f(x)\geq0\qquad\text{for a.e.}\quad x\in\R$$
    $$f\in H^k(\R)\cap W^{1,1}(\R)$$
    $$\|f_0\|_{L^1}=1$$
    for some $k\geq1$. Here we use the usual notation for Sobolev spaces $W^{k,p}(\R)$ and Hilbert spaces $H^k(\R)=W^{k,2}(\R)$,
    \item[(A2)] the interaction function $\phi(x,y)$ is a gradient $\phi(x,y)= \Phi' (x-y)$, $\phi\in L^2(\R)\cap L^\infty(\R)$  and that there exists a function $h\in L^p(\R)$ for some $p\in[1,2]$ and a constant $b\in\R$ such that $\Phi''=h+b\delta_0$, where this identity has to be understood in the sense of distributions.
\end{itemize}
\end{hyp}
Under these hypothesis we have that, following \cite{laurent2007local}, for all $ T>0$ the solution $f_t(x)$ satisfies the following properties:
$$f\in L^\infty([0,T];H^k(\R)\cap W^{1,1}(\R)),\qquad\partial_tf\in L^\infty([0,T];H^{k-1}(\R))$$
$$f(x,0)=f_0(x);\qquad f(x,t)\geq0 \quad\text{for a.e.}\quad (t,x)\in[0,T]\times\R$$
$$\|f(\cdot,t)\|_{L^1}=1\qquad\text{for a.e.}\quad t\in[0,T]$$
The regularity of $f_t$ and $\partial_tf_t$ implies that $f_t(x)\in C^0([0,T];H^{k-1}(\R))$, see for example \cite[Lemma~1.2, page 7]{lions1969quelques}. Therefore the evaluation at $t=0$ makes sense. Moreover, due to classical continuous Sobolev embeddings, we have that $f_t(x)\in C^0([0,T];C^{k-1,\frac{1}{2}}(\R))$, where $C^{k,\gamma}$ is the space of functions having continuous derivatives up through order $k$ and such that the $k$-th partial derivatives are Hölder continuous with exponent $\alpha$.
In particular even the spatial evaluation makes sense.\\

We would like to emphasize that the previous set of hypotheses is only used to ensure the existence and uniqueness of a solution with these properties. The proof given in \cite{laurent2007local} for the case of $\R^d$ can be extended to periodic domains by standard techniques, which are among the cases we are interested in.
\begin{rk}
    We would like to quickly notice that on $\R$ the supposing that $\phi$ is a gradient of a function $\Phi$ is really easy to satisfy. Indeed is sufficient that for $\phi(x,y)=\phi(x-y)$ we have that $\phi\in C^0(\R)$, due to the fundamental theorem of calculus.
    On the other hand, if $\phi$ acts on $\R^d$ for $d\geq1$ and it is not a gradient then, supposing $\phi$ sufficiently regular, e.g., bounded Lipschitz and $C^1$, and the initial datum is $C^1$ then one obtains $C^1$ solutions, as remarked in \cite{lancellotti2005vlasov}. 
\end{rk}
 We are finally ready to state and prove our proposition.
\begin{proposition}\label{prop1}
    Assume the previous hypothesis and define $x_t(\xi)$ as
    $$x_t(\xi)\coloneqq F^{-1}_t(\xi)$$
    with $F^{-1}_t$ the inverse\footnote{We have omitted the dependency on the measure for brevity, so we identify $F_{\mu_t}$ with $F_{t}$ and the same holds for $F^{-1}_t$} of the CDF relative to $f_t$, solution of \eqref{MFL}. Then we have that $x_t(\xi)$ solves
    $$\partial_t x(t,\xi)=\int_I\phi(x(t,\xi),x(t,z))\textnormal{d}z\qquad\text{for a.e. }t\in[0,T].$$
\end{proposition}
\begin{proof}    
The first thing to prove is that 
$$\partial_t F_t(x)=-f_t(x)\int_I\phi(x,x(t,\xi))\textnormal{d}\xi$$
in the scalar distribution sense on $[0, T]$. That is because 
$$\partial_t f_t(x)+\partial_x\left(f_t(x)\int_\R \phi(x,y)f_t(y)\;\textnormal{d}y\right)=0$$
so, integrating from $-\infty$ to $x$ we obtain,
$$0=\partial_t F_t(x)+f_t(x)\int_\R \phi(x,y)f_t(y)\;\textnormal{d}y=\partial_t F_t(x)+f_t(x)\int_I \phi(x,x(t,\xi))\textnormal{d}\xi$$
where we have used that $\lim_{x\rightarrow-\infty}f_t(x)\int_\R \phi(x,y)f_t(y)\;\textnormal{d}y=0$, since $\phi$ is bounded and $f_t(x)\in C^0([0,T];C^{k-1,\frac{1}{2}}(\R))$, so that $$\lim_{x\rightarrow\pm\infty}f_t(x)=0\qquad\forall \;t\in[0,T].$$
Moreover 
$$\partial_tF_t(x)=\int_{-\infty}^x\partial_tf_t(y)\textnormal{d}y$$
has to be interpreted in the scalar distribution sense on $[0, T]$, see \cite[ Lemma 1.2, ch. III, sec. 1]{temam1984navier}.
Then using that $x(t,\xi)$ is the inverse of $F_t(x)$ we have that
$$F_t(x(t,\xi))=\xi$$
and differentiating such expression we obtain
$$\partial_t x(t,\xi)=-\frac{\partial_t F_t(x(t,\xi))}{\partial_x F_t(x(t,\xi))}=\frac{f_t(x(t,\xi))\int_I\phi(x(t,\xi),x(t,z))\textnormal{d}z}{f_t(x(t,\xi))}=\int_I\phi(x(t,\xi),x(t,z))\textnormal{d}z$$
in a scalar distributional sense on $[0,T]$. In the last equality we have used that
$$\int_\R\phi(x,y)f_t(y)\textnormal{d}y=\int_I\phi(x,x(t,\xi))\textnormal{d}\xi$$ 
and that is because $f_t$ can be written as the pushforward of the restriction of the one-dimensional Lebesgue measure restricted to $I$ by $F^{-1}_t$, i.e. $f_t(x)\textnormal{d}x=(F^{-1}_t)_\#(\mathcal{L}^1\resmes[0,1])$, see \cite[Prop.~2.2.]{santambrogio2015optimal}. This completes our proof.
\end{proof}
An important remark is that the measurable function $x(\xi)$ obtained with (\ref{transf}) is in general not unique. Any measure-preserving rearrangement of $x(\xi)$ gives the same measure $\mu$. For example consider 
$\mu=\frac{1}{3}\delta_{-1}+\frac{2}{3}\delta_{1}$ then, by our standard choice of $x(\xi)$ we would obtain 
$$x(\xi)=\begin{cases}
    -1\qquad \text{if }\xi\in\left(0,\frac{1}{3}\right],\\
    1\qquad \text{if }\xi\in\left(\frac{1}{3},1\right]
\end{cases}$$
but it is clear that another suitable choice could be
$$\Tilde{x}(\xi)=\begin{cases}
    -1\qquad \text{if }\xi\in\left(\frac{1}{3},\frac{2}{3}\right],\\
    1\qquad \text{if }\xi\in\left(0,\frac{1}{3}\right]\cup\left(\frac{2}{3},1\right]
\end{cases}$$
or $\Bar{x}(\xi)=x(\xi+c)$ for all $c\in\R$. This phenomenon of non-uniqueness is the reason why the \eqref{MFL} contains less detailed information compared to \eqref{CL}: many solutions of the first correspond to the same solution of second, leading to a loss of information. This is quite natural as the continuum limit \eqref{CL} still contains a heterogeneous tracking variable for each node/agent.\\

As we already mentioned, we are particularly interested in periodic phase space given by the circle $\mathbb{S}^1$ and from now on we will consider only this setting. This is because the Kuramoto model has as variables the phases of the oscillators, which belong naturally to $\mathbb{S}^1$. For this reason we have to reinterpret the following pages and the previous proposition has to be adjusted due to technical reasons.
\begin{definition}
We define the \textit{circular} cumulative distribution function (CCDF) starting from $a\in[0,2\pi]$ of the measure $\mu$ with density $f$ as 
$$F_{\mu,a}:\mathbb{S}^1\equiv[a,2\pi+a]/\sim \longrightarrow[0,1]/\sim$$
    $$x\longrightarrow F_{\mu,a}(x)=\int_a^xf(y)\;\textnormal{d}y$$
 where $\sim$ is the equivalence relation that identifies the endpoints of the interval.\end{definition}
 We had the need to identify even $I=[0,1]$ because we wanted $F_{\mu,a}$ to be well defined and its inverse to be periodically defined. Notice that the continuum limit PDE \eqref{CL} admits solutions even if the domain is periodic, by standard arguments, following the one contained in the seminal paper \cite{medvedev2014nonlinear}.\\

 In the following, we will always consider $a=0$ and we will denote $F_{\mu,0}\equiv F_\mu$.
Then, concerning the periodic case, we have the following result:
\begin{proposition}\label{prop2}
     Assume the previous hypothesis on $\mathbb{T}$ instead of $\R$ and define $x_t(\xi)$ as 
     $$\Tilde{x}_t(\xi)=F^{-1}_{t}(\xi+C(t))$$
     with $\xi\in[0,1]$ and $C(t)=\int_a^tf_s(0)\int_0^{2\pi}\phi(0,y)f_s(y)\;\textnormal{d}y\;\textnormal{d}s$. Moreover $F^{-1}_t$ is again the inverse of the CDF relative to $f_t$, solution of \eqref{MFL}. Then we have that $x_t(\xi)$ solves
    $$\partial_t \Tilde{x}(t,\xi)=\int_I\phi(\Tilde{x}(t,\xi),\Tilde{x}(t,z))\textnormal{d}z\qquad\text{for a.e. }t\in[0,T].$$
\end{proposition}
\begin{proof}
We start by defining 
    $$\Tilde{F}_t(x)\coloneqq F_t(x)-C(t)$$
so that, as before
$$\partial_t \Tilde{F}_t(x)=-f_t(x)\int_I\phi(x,x(t,\xi))\textnormal{d}\xi$$
in the scalar distribution sense on $[0, T]$. Here we have used the presence of the extra term $C(t)$ in order to cancel the boundary term obtained by the integration. Moreover using that $\partial_x \Tilde{F}_t=\partial_x F_t$ we obtain
$$\partial_t \Tilde{x}(t,\xi)=\int_I\phi(\Tilde{x}(t,\xi),x(t,z)) \textnormal{d}z=\int_I\phi(\Tilde{x}(t,\xi),\Tilde{x}(t,z-C(t))) \textnormal{d}z=\int_I\phi(\Tilde{x}(t,\xi),\Tilde{x}(t,z)) \textnormal{d}z$$
where for the last equality we have used the periodicity of $\Tilde{x}(\xi)$.
\end{proof}
Note that for both previous propositions, if we have that $\phi$ is also Lipschitz continuous then one obtain by classical argument the uniqueness of a strong and weak solution for \eqref{CL}, therefore implying that our new distributional solutions can be interpreted as classical solutions.\\

Before proceeding with the next main result of the paper, we would like to briefly discuss two natural questions that arise from our previous propositions:
\begin{itemize}
    \item Is it possible to extend the previous results to $\R^d$ for $d\geq2$?
    \item What happens if we consider a system of non-exchangeable particles? Is it possible to generalize the latter formula in this case?
\end{itemize}
Regarding the first question, we believe that it is not possible to extend such a formula to $d \geq 2$, at least not in the simple way we have presented. There are several hints supporting this intuition. The first is that, in general, the set of measures
$$\left\{\;\int_I\delta_{x_0(\xi)}(x)\textnormal{d}\xi\;\bigg\vert\;x_0:I\rightarrow\R^d\;\text{measurable}\;\right\}$$
is just dense in $\mathcal{P}(\R^d)$ for the weak topology when $d\geq2$, which follows easily from the density of finite linear combination of Dirac's deltas and, furthermore, it is not possible to identify the function $x_0(\xi)$ as the quantile function of our measure. Indeed, the pseudo-inverse \eqref{transf} is not well-defined in the multidimensional case: such an infimum would not be unique, since for $d \geq 2$ this operation would identify an entire level set in $\R^d$.

Another topological argument against the possibility of such an extension is the following. If we consider a measure $\mu \in \mathcal{P}(\R^d)$ with a smooth regular density $f$ and its cumulative distribution function $F_\mu$, we would expect the pseudo-inverse to coincide with the true inverse of $F_\mu$, as shown before. However, this would define a homeomorphism between spaces of different dimensions, which is impossible. Hence, we would be forced to restrict our attention to pathological or singular densities, which already makes such an extension less useful and interesting.\\

We now turn our attention to the second question, that it has been already partially answered. Indeed if $\phi=\phi(\xi,\zeta,x,y)$, starting from a solution $x_t(\xi)$ of \eqref{CL} we can still obtain a solution $\nu_t(\xi,x)\in C([0,T];\mathcal{P}(I\times\R^d))$ of the non-exchangeable mean-field 
\begin{equation}\label{NMFL}
    \begin{cases}\displaystyle
        \partial_t \nu_t(\xi,x) +\nabla_x\cdot\left(\left(\int_{I\times\R^d}\phi(\xi,\zeta,x,y)\textnormal{d}\nu_t(\zeta,y)\textnormal{d}\zeta\right)\nu_t(\xi,x)\right)=0,\\
        \nu_{t=0}=\nu_0,
    \end{cases}
\end{equation}
in an analogous way of \ref{transf2}, precisely:
$$\nu_t(\xi,x)=\int_I\delta_{x(t,\zeta)}(x)\delta_\zeta(\xi)\textnormal{d}\zeta.$$
The other direction, i.e. from \eqref{NMFL} to the \eqref{CL}, is much more challenging and still open. Indeed a general answer has been given only if the interaction function can be written as the
product of a function of $\xi$, $\zeta$ and of a linear function of $x$, $y$, i.e.,
$$\phi(\xi,\zeta,x,y)=\Tilde{\phi}(\xi,\zeta)(\lambda_1x+\lambda_2y)$$
with $\Tilde{\phi}:I^2\rightarrow\R$ and $(\lambda_1,\lambda_2)\in\R^2$. This form is actually common in models for opinion dynamics with linear interaction of the type Hegselmann-Krause, see \cite{hegselmann2015opinion}, but for the general non-linear case there is still no answer, see \cite{paul2022microscopic} for a complete discussion.\\

Now that we have proved and discussed our propositions regarding solutions, it is natural to try to utilize them to known objects of \eqref{MFL} in order to get new ones for the \eqref{CL}. Such objects could be geometrical structures such as invariant manifolds. So far, the question how to link the mean-field and continuum limit invariant manifolds was completely open. Below we are going to prove, how to transform the most studied and remarkable invariant manifold of \eqref{MFL}, namely the Ott-Antonsen manifold to an unstable manifold for the continuum limit.\jump\jump
\textbf{\textit{Structure of the unstable invariant manifold
of the mean-field equation.}}
When we consider \eqref{MFL-Kura} with $\omega$ set to zero and, for simplicity $K=1$, we obtain the mean-field limit equation of a system of identical phase oscillators
\begin{equation}\label{hom-MFL-Kura}
    \partial_t f(\theta,t)=\partial_\theta\left[f(\theta,t)\left(\int_0^{2\pi} f(\phi,t)\sin(\theta-\phi)\textnormal{d}\phi\right)\right].
\end{equation}
For this particular form of the equation the Ott-Antonsen (OA) manifold can be explicitly written as follows
\begin{align*}
        \begin{split}
            \mathcal{M}_{OA}=&\left\{\frac{1}{2\pi}+\sum_{n=0}^\infty \beta^n\cos(n(\theta+\alpha))\;\bigg\vert\;\beta\in[0,1),\;\alpha\in[-\pi,\pi]\right\}\\
            =&\left\{\frac{1}{2\pi}\frac{1-\beta^2}{1-2\beta\cos(\alpha+\theta)+\beta^2}\;\bigg\vert\;\beta\in[0,1),\;\alpha\in[-\pi,\pi]\right\}.
        \end{split}
    \end{align*}
It is then a standard proof, following ideas working even for non-identical oscillators \cite{ott2008low,ott2009long}, to show that the dynamics on the OA manifold is given by the ODEs
$$\begin{cases}
        \dot{\beta}=\frac{1}{2}\beta(1-\beta^2),\\
        \dot{\alpha}=0.
    \end{cases}$$
The ODEs can be explicitly solved. First, we find $\alpha(t)=\alpha(0)\in[-\pi,\pi]$ and then one directly verifies that the first ODE is solved by
$$\beta(t)=\frac{\beta(0)}{\sqrt{\beta(0)^2+(1-\beta(0)^2)e^{-t}}}.$$
From this explicit expression we can deduce that as $t\rightarrow-\infty$ we converge towards the incoherent equilibrium $\frac{1}{2\pi}$, which corresponds to $\beta=0$. Linearizing \ref{hom-MFL-Kura} around this equilibrium we obtain the linear operator
$$L_{\frac{1}{2\pi}}f(\theta)\coloneqq\frac{1}{2\pi}\int_0^{2\pi}\cos(\theta-\phi)f(\phi)$$
that on $L^2(\mathbb{S}^1)$ has
\begin{itemize}
    \item 2 unstable directions, given by $\sin(\theta)$ and $\cos(\theta)$, with eigenvalue $\frac{1}{2}$,
    \item infinitely many central directions, given by the remaining elements of the trigonometric basis, namely $\sin(n\theta)$, $\cos(n\theta)$ for $n\geq2$ and $n=0$. These directions all have eigenvalue equal to $0$.
\end{itemize}
So, by a standard invariant manifold theorem, see \cite{sell2013dynamics} or \cite{henry2006geometric}, there exists a two-dimensional unstable manifold $W^\u(\frac{1}{2\pi})$ of the mean-field limit that can be identified with $\mathcal{M}_{OA}$, in view of the previous analysis and uniqueness of the unstable manifold \cite[Lemma 71.2]{sell2013dynamics}. From now on we will denote the typical element of this manifold with $f_{\alpha,\beta}(\theta)$, where the indexes will be omitted if not explicitly needed.

As discussed above, we want to apply Proposition \ref{prop2} to the $\mathcal{M}_{OA}$ and obtain a new invariant manifold of the continuum limit \eqref{CL}. In particular, the continuum limit PDE that we work with is
\begin{equation}\label{CL-Kura}
    \partial_t x(t,\xi)=\int_0^1 \sin\big(x(t,\xi)-x(t,z)\big)\textnormal{d}z.
\end{equation}
In order to do what we have planned, we claim that the closed form of the CCDF of an element of $\mathcal{M}_{OA}$ is computable as
    \begin{align*}
        \begin{split}
            F_{\alpha,\beta}(\theta)=&\int_0^\theta f_{\alpha,\beta}(\phi)\;\textnormal{d}\phi=\frac{\theta}{2\pi}+\sum_{n=1}^\infty \frac{\beta^n}{n}\big(\sin(n(\theta+\alpha)-\sin(n\alpha)\big)=\\
            =& \frac{\theta}{2\pi}+\frac{1}{\pi}\left(\arctan\left(\frac{\beta\sin(\theta+\alpha)}{1-\beta\cos(\theta+\alpha)}\right)-\arctan\left(\frac{\beta\sin(\alpha)}{1-\beta\cos(\alpha)}\right)\right)
        \end{split}
    \end{align*}
where we have interchanged integration and summation using the
dominated convergence theorem. Moreover, we have used that, for $|x|<1$, we have the following equalities
$$\sum_{n=1}^\infty\frac{x^n}{n}\sin(n\psi)=\textnormal{Im}\left(\sum_{n=1}^\infty\frac{x^n}{n}\e^{\i n\psi}\right)=\textnormal{Im}\left(-\ln\left(1-x\e^{\i\psi}\right)\right)=\arctan\left(\frac{x\sin\psi}{1-x\cos\psi}\right),$$
because of the series expansion of the logarithm and the last equality is due to its definition in $\C$. It is possible to prove that the inverse of $F_{\alpha,\beta}$ has a closed form too, and it is given by
$$F^{-1}_{\alpha,\beta}(\xi)=\begin{cases}
        G_{\alpha,\beta}(\xi)\qquad\qquad\; \text{if}\;\xi\in[0,c]\\
        G_{\alpha,\beta}(\xi)+2\pi\qquad \text{if}\;\xi\in[c,1]
    \end{cases}$$
where we have that
$$G_{\alpha,\beta}(\xi)\coloneqq2\arctan\left(\frac{1-\beta}{1+\beta}\tan\left(\pi\xi+\arctan\left(\frac{1+\beta}{1-\beta}\tan\left(\frac{\alpha}{2}\right)\right)\right)\right)-\alpha$$
    $$c\coloneqq\frac{1}{2}-\frac{1}{\pi}\arctan\left(\frac{1+\beta}{1-\beta}\tan\left(\frac{\alpha}{2}\right)\right).$$
The interested reader can find the detailed computations in the appendix. Since the function's output is on a periodic domain, the vertical shift of $2\pi$ at $c$ does not change the final point. However, we apply this shift to create a function that is continuous on the standard interval $[0,2\pi]\subset\R$ without needing to consider the endpoints as identical.\\

Before proceeding with the main theorem, we would like to see how the integral vector field of \eqref{CL} acts on a general function $F_{\alpha,\beta}^{-1}(\xi)$. What we obtain is the following lemma.
\begin{lemma}\label{lemma}
Consider the function $F^{-1}_{\alpha,\beta}(\xi)$, previously defined. Then $\forall \alpha\in[-\pi,\pi],\forall\beta\in[0,1)$ and $\forall\xi\in[0,1]$ the following equality holds:
$$\int_0^1\sin\left(F^{-1}_{\alpha,\beta}(z)-F^{-1}_{\alpha,\beta}(\xi)\right)\;\textnormal{d}z=-\beta \sin\left(\alpha+F^{-1}_{\alpha,\beta}(\xi)\right)$$
\end{lemma}
\begin{proof}
Let us start with the natural change of variable $F^{-1}_{\alpha,\beta}(z)=u$ so that we obtain
\begin{align*}
    \begin{split}
    \int_0^1\sin\left(F^{-1}_{\alpha,\beta}(z)-F^{-1}_{\alpha,\beta}(\xi)\right)\;\textnormal{d}z=&\int^{2\pi}_0\sin{\left(u-v\right)}f_{\alpha,\beta}(u)\textnormal{d}u\\
    =&\;\frac{1-\beta^2}{2\pi}\int^{2\pi}_0\frac{\sin{\left(u-v\right)}}{1-2\beta\cos(\alpha+u)+\beta^2}\textnormal{d}u
    \end{split}
\end{align*}
where, for notational convenience, we have denoted $v=F^{-1}_{\alpha,\beta}(\xi)$. The most elegant way to compute the integral on the right hand side is considering it as a complex integral. Another way of doing it could be to use the series expansion of $f_{\alpha,\beta}$ and take advantage of the orthogonality of the trigonometric basis. So using that $\sin(u-v)=\textnormal{Im}(\e^{\i(u-v)})$ we just have to compute the following integral
$$\int^{2\pi}_0\frac{\e^{\i u}}{1-2\beta\cos(\alpha+u)+\beta^2}\textnormal{d}u.$$
We will now compute the contour integral imposing the complex substitution $z=\e^{\i u}$ obtaining
$$\int^{2\pi}_0\frac{\e^{\i u}}{1-2\beta\cos(\alpha+u)+\beta^2}\textnormal{d}u=\frac{1}{\i}\int_{|z|=1}\frac{z}{(1-\beta \e^{\i\alpha}z)(z-\beta \e^{-\i\alpha})}\textnormal{d}z$$
where we have used the equality $$1-2\beta\cos(\alpha+u)+\beta^2=(1-\beta \e^{\i(\alpha+u)})(1-\beta \e^{-\i(\alpha+u)}).$$
Now the integrand has two simple poles, namely $z_1=\beta \e^{-\i\alpha}$ and $z_2=(\beta \e^{\i\alpha})^{-1}$ but only $z_1$ is inside the unit disc, since $\beta<1$. So using the residue theorem we obtain that
$$\frac{1}{\i}\int_{|z|=1}\frac{z}{(1-\beta \e^{\i\alpha}z)(z-\beta \e^{-\i\alpha})}\textnormal{d}u=2\pi \frac{z}{1-\beta \e^{\i\alpha}z}\bigg\vert_{z=z_1}=2\pi\frac{\beta \e^{-\i\alpha}}{1-\beta^2}$$
and finally, putting all together we obtain 
\begin{align*}
    \begin{split}
        &\int_0^1\sin\left(F^{-1}_{\alpha,\beta}(z)-F^{-1}_{\alpha,\beta}(\xi)\right)\;\textnormal{d}z=\frac{1-\beta^2}{2\pi}\textnormal{Im}\left(e^{-\i v}\int^{2\pi}_0\frac{\e^{\i u}}{1-2\beta\cos(\alpha+u)+\beta^2}\textnormal{d}u\right)=\\
        =&\frac{1-\beta^2}{2\pi}\textnormal{Im}\left(\e^{-\i v}2\pi\frac{\beta \e^{-\i\alpha}}{1-\beta^2}\right)=\beta\textnormal{Im}\left(\e^{-\i(\alpha+v)}\right)=-\beta\sin\left(\alpha+F^{-1}_{\alpha,\beta}(\xi)\right)
    \end{split}
\end{align*}
as desired.
\end{proof}
The next theorem is about an explicit characterization for the invariant manifold of the so called \textit{twisted state} $y_q(\xi)\coloneqq2\pi\xi+q$ with $q\in\mathbb{T}$, equilibria of the system (\ref{CL-Kura}). Such equilibria arise as the continuum limit of their discrete N-dimensional counterpart
$$y^N_q=\left(0,\frac{2\pi}{N}+q,\ldots,\frac{2\pi}{N}(N-1)+q\right)$$
and their stability have been studied in \cite{wiley2006size}. Actually, the states $y_q$ are just a particular case of the general twisted states, defined as $y_{q,m}(\xi)=2\pi\xi m+q$, where the parameter $m\in\Z$ accounts for the number of twists the solution makes around the circle as $\xi$ varies from $0$ to $1$. Their properties and non-linear stability has been investigated in \cite{bick2023phase,medvedev2015stability}. Furthermore their role in the Kuramoto model on small-world graphs has been studied in \cite{medvedev2014small}.

Moreover this whole family of equilibria can be obtained from the incoherent equilibrium $\frac{1}{2\pi}$ using the transformation defined in \ref{prop2} and one can obtain the latter equilibrium applying \eqref{transf2} to $y_q(\xi)$. We are now ready to state and prove our main theorem.
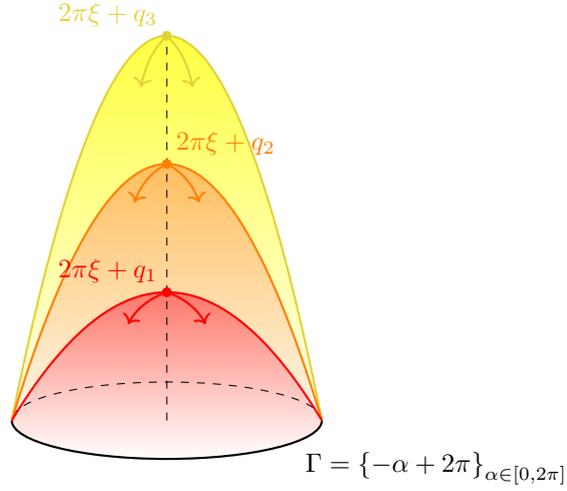
\begin{figure}[h]
\centering
\hspace*{3.3cm}\begin{tikzpicture}[scale=1.7]
\def\a{1.2} 
\def\b{0.3} 
\fill[white, shading=radial, top color=yellow!100, bottom color=yellow!00,opacity=0.75]
  plot[domain=-\a:\a,samples=100] (\x, {(3*(1 - (\x/\a)^2))})-- (\a,0) arc[start angle=0, end angle=-180, x radius=\a, y radius=\b]-- cycle;
\fill[white, shading=radial, top color=orange!100, bottom color=orange!0, opacity=0.5]
  plot[domain=-\a:\a,samples=100] (\x, {(2*(1 - (\x/\a)^2))}) -- (\a,0) arc[start angle=0, end angle=-180, x radius=\a, y radius=\b] -- cycle; 
\fill[white, shading=radial, top color=red!100, bottom color=red!0, opacity=0.5]
  plot[domain=-\a:\a,samples=100] (\x, {(1*(1 - (\x/\a)^2))}) -- (\a,0) arc[start angle=0, end angle=-180, x radius=\a, y radius=\b] -- cycle;
\draw[dashed] (\a,0) arc [start angle=0,end angle=180,x radius=\a,y radius=\b];
\draw[thick] (-\a,0) arc [start angle=180,end angle=360,x radius=\a,y radius=\b];
\draw[thick, yellow!85!black]
  plot[domain=-1*\a:1*\a,samples=100] (\x, {(3*(1 - 1*(\x/\a)^2))});
\draw[thick, orange]
  plot[domain=-1*\a:1*\a,samples=100] (\x, {(2*(1 - 1*(\x/\a)^2))});
\draw[thick, red]
  plot[domain=-\a:\a,samples=100] (\x, {((1 - (\x/\a)^2))});
\draw[dashed] (0,0) -- (0,3);
\node[yellow!85!black, above left] at (0,3) {$2\pi\xi+q_3$};
\node[orange, above right] at (0,2) {$2\pi\xi+q_2$};
\node[red, above left] at (0,1) {$2\pi\xi+q_1$};
\node[below right] at (1,-5pt) {$\Gamma=\left\{-\alpha+2\pi\right\}_{\alpha\in[0,2\pi]}$};
\filldraw[yellow!85!black, fill=yellow!85!black, thick] (0,3) circle (0.03);
\filldraw[orange, fill=orange, thick] (0,2) circle (0.03);
\filldraw[red, fill=red, thick] (0,1) circle (0.03);
\coordinate (A1) at (0,3);

\coordinate (B1) at (0.20,2.6);
\coordinate (B2) at (-0.20,2.6);

\draw[->, yellow!85!black, thick]
  (A1) .. controls (0.05,2.95) and (0.15,2.85) .. (B1);

\draw[->, yellow!85!black, thick]
  (A1) .. controls (-0.05,2.95) and (-0.15,2.85) .. (B2);

\coordinate (A1) at (0,2);

\coordinate (B1) at (0.25,1.7);
\coordinate (B2) at (-0.25,1.7);

\draw[->, orange, thick]
  (A1) .. controls (0.10,1.95) and (0.20,1.85) .. (B1);

\draw[->, orange, thick]
  (A1) .. controls (-0.10,1.95) and (-0.2,1.85) .. (B2);

\coordinate (A1) at (0,1);

\coordinate (B1) at (0.30,0.75);
\coordinate (B2) at (-0.30,0.75);

\draw[->, red, thick]
  (A1) .. controls (0.15,0.95) and (0.25,0.85) .. (B1);

\draw[->, red, thick]
  (A1) .. controls (-0.15,0.95) and (-0.25,0.85) .. (B2);

\end{tikzpicture}
\caption{Ideal graphic representation of the unstable manifolds $W^\u(y_q(\xi))$, for three different values of $q$, and their common asymptotic curve $\Gamma$.}
\label{fig:1}
\centering
\end{figure}
\begin{theorem}
    The family of incoherent equilibria $y_q(\xi)\coloneqq2\pi\xi+q$ of (\ref{CL-Kura}), with $q\in\mathbb{T}$ have an explicit expression for their 2-dimensional unstable manifolds, given by
    $$W^\u(y_q(\xi))=\left\{F^{-1}_{\alpha,\beta}\big(\xi+C(\alpha,\beta)+q\big)\bigg\vert\;\beta\in[0,1),\;\alpha\in[-\pi,\pi]\right\}$$
    with
    $$C(\alpha,\beta)=-\frac{1}{\pi}\arctan\left(\frac{\beta\sin(\alpha)}{1-\beta\cos(\alpha)}\right).$$
    Moreover, the parameters $(\alpha,\beta)$ that define the two manifolds, satisfy the same system of ODEs, namely
    $$\begin{cases}
        \dot{\beta}=\frac{1}{2}\beta(1-\beta^2),\\
        \dot{\alpha}=0.
    \end{cases}$$
\end{theorem}
\begin{proof}
    We start by noticing the really simple fact
    $$F^{-1}_{\alpha,0}\big(\xi+C(\alpha,0)+q\big)=y_q(\xi)\qquad\forall\xi\in[0,1],\alpha\in[-\pi,\pi],$$
    so that the equilibria belong indeed to our manifolds. Moreover, Lemma \ref{lemma} holds identically up to translations, due to the periodicity of $F^{-1}_{\alpha,\beta}$. So that, by plugging in the expression of the typical element of $W^\u(y_q(\xi))$ in \eqref{CL-Kura} and differentiating we obtain
    \begin{align}\label{eq:1}
        \begin{split}
            -&\beta \sin\left(\alpha+\Theta_{\alpha,\beta}(\xi)\right)=\\
            &=\dot{\alpha}\partial_\alpha\Theta_{\alpha,\beta}(\xi)+\dot{\beta}\partial_\beta\Theta_{\alpha,\beta}(\xi)+\partial_\xi\Theta_{\alpha,\beta}(\xi)\left[\dot{\alpha}\partial_\alpha C(\alpha,\beta)+\dot{\beta}\partial_\beta C(\alpha,\beta)\right]=\\
            &=\dot{\alpha}\left(-\frac{\partial_\alpha F_{\alpha,\beta}(\Theta_{\alpha,\beta})}{\partial_x F_{\alpha,\beta}(\Theta_{\alpha,\beta})}\right)+\dot{\beta}\left(-\frac{\partial_\beta F_{\alpha,\beta}(\Theta_{\alpha,\beta})}{\partial_x F_{\alpha,\beta}(\Theta_{\alpha,\beta})}\right)+\frac{1}{\partial_x F_{\alpha,\beta}(\Theta_{\alpha,\beta})}\left[\dot{\alpha}\partial_\alpha C+\dot{\beta}\partial_\beta C\right],
        \end{split}
    \end{align}
    where we have denoted $\Theta_{\alpha,\beta}(\xi)=F^{-1}_{\alpha,\beta}(\xi+C(\alpha,\beta)+q)$. By straightforward but not quick computations one can prove that the following equalities hold:
    $$\partial_\theta F_{\alpha,\beta}(\theta)=f_{\alpha,\beta}(\theta),$$
    $$\partial_\alpha F_{\alpha,\beta}(\theta)=\frac{1}{\pi}\left(\frac{\beta\cos(\theta+\alpha)-\beta^2}{1-2\beta\cos(\theta+\alpha)+\beta^2}-\frac{\beta\cos(\alpha)-\beta^2}{1-2\beta\cos(\alpha)+\beta^2}\right),$$
    $$\partial_\beta F_{\alpha,\beta}(\theta)=\frac{1}{\pi}\left(\frac{\sin(\theta+\alpha)}{1-2\beta\cos(\theta+\alpha)+\beta^2}-\frac{\sin(\alpha)}{1-2\beta\cos(\alpha)+\beta^2}\right)$$
    and using very similar computations it follows also that 
    $$\partial_\alpha C(\alpha,\beta)=-\frac{1}{\pi}\frac{\beta\cos(\alpha)-\beta^2}{1-2\beta\cos(\alpha)+\beta^2},$$
    $$\partial_\beta C(\alpha,\beta)=-\frac{1}{\pi}\frac{\sin(\alpha)}{1-2\beta\cos(\alpha)+\beta^2}.$$
    Substituting such expressions in \eqref{eq:1} we obtain 
    \begin{align*}
        \begin{split}
            -\beta\sin&\left(\alpha+\Theta_{\alpha,\beta}\right)=\frac{-1}{\pi f_{\alpha,\beta}(\Theta_{\alpha,\beta})}\left[\frac{\beta\cos(\Theta_{\alpha,\beta}+\alpha)-\beta^2}{1-2\beta\cos(\Theta_{\alpha,\beta}+\alpha)+\beta^2}\dot{\alpha}+\frac{\sin(\Theta_{\alpha,\beta}+\alpha)}{1-2\beta\cos(\Theta_{\alpha,\beta}+\alpha)+\beta^2}\dot{\beta}\right]\\
            =&\frac{-2}{1-\beta^2}\left[(\beta\cos(\alpha+\Theta_{\alpha,\beta})-\beta^2)\dot{\alpha}+\sin(\alpha+\Theta_{\alpha,\beta})\dot{\beta}\right]
        \end{split}
    \end{align*}
    so that, by choosing 
    $$\begin{cases}
        \dot{\beta}=\frac{1}{2}\beta(1-\beta^2),\\
        \dot{\alpha}=0,
    \end{cases}$$
    we have proved our statement.
\end{proof}
\section{Conclusion}
In this work, we have established a first direct dynamical correspondence between invariant manifolds of the mean-field and continuum limits of interacting particle systems (IPS). Specifically, in the particular case of Kuramoto-type systems, we have shown that the unstable manifold of the incoherent steady state of the continuum limit \eqref{CL} equation can be obtained explicitly from the celebrated Ott–Antonsen manifold of the mean-field limit \eqref{MFL}. This result provides a concrete example of how invariant geometric structures in different large-population limits can be dynamically related, extending the existing understanding of the OA manifold and its analogues.\\

It is remarkable how there are dynamical analogies between the two manifolds. From our analysis we can deduce that $W^\u(y_q(\xi))$ can not be a global attractor for any $q\in\R$. Indeed, all the manifolds share the same dynamical properties, making it impossible. This is related to the fact that the Ott-Antonsen manifold $\mathcal{M}_{OA}$ can be at best neutrally stable in the case of identical oscillators, see \cite{engelbrecht2020ott}. Furthermore, as $\mathcal{M}_{OA}$ is made out of orbits that connect the incoherent state $\frac{1}{2\pi}$ to the synchronous states $\delta_{\alpha_0}$ for some $\alpha_0\in[0,2\pi]$ as $\beta\rightarrow1$, the same happens on $W^\u(y_q(\xi))$. Once one fixes $q\in[0,2\pi]$, the orbits of $W^\u(y_q(\xi))$ connect the incoherent state $2\pi\xi+q$ to the synchronous states $-\alpha+2\pi$ as $\beta\rightarrow1$, independently from $q$, as illustrated in \ref{fig:1}.\\

Our work is part of the broader challenge to develop a unified framework for identifying, comparing, and linking invariant manifolds across the various limiting descriptions of IPS, from microscopic systems to mean-field and continuum PDEs. In particular, we plan to develop more general procedures to study the convergence of the invariant manifolds of finite $N$-dimensional system as $N\rightarrow+\infty$ towards the ones of the mean-field and continuum PDEs. A natural next step is to investigate, analytically and numerically, whether the same transformation applied to other invariant sets of the mean-field limit produces geometrical objects with analogous dynamical properties for the \eqref{CL} PDE.
\bibliographystyle{abbrv}
\bibliography{zReference}
\section{Appendix: Closed formula for $F^{-1}_{\alpha,\beta}(\xi)$}
In the following we would like to justify the explicit form of the inverse of $F_{\alpha,\beta}(\theta): [0,2\pi]\rightarrow[0,1]$, which is 
 $$F^{-1}_{\alpha,\beta}(\xi)=\begin{cases}
        G_{\alpha,\beta}(\xi)\qquad\qquad\; \text{if}\;\xi\in[0,c]\\
        G_{\alpha,\beta}(\xi)+2\pi\qquad \text{if}\;\xi\in(c,1]
    \end{cases}$$
    where we have that
    $$G_{\alpha,\beta}(\xi)\coloneqq2\arctan\left(\frac{1-\beta}{1+\beta}\tan\left(\pi\xi+\arctan\left(\frac{1+\beta}{1-\beta}\tan\left(\frac{\alpha}{2}\right)\right)\right)\right)-\alpha$$
    $$c\coloneqq\frac{1}{2}-\frac{1}{\pi}\arctan\left(\frac{1+\beta}{1-\beta}\tan\left(\frac{\alpha}{2}\right)\right).$$
So, given $\xi\in[0,1]$ we want to $\theta\in[0,2\pi]$ such that
$$\xi=F_{\alpha,\beta}(\theta)=\frac{\theta}{2\pi}+\frac{1}{\pi}\left(\arctan\left(\frac{\beta\sin(\theta+\alpha)}{1-\beta\cos(\theta+\alpha)}\right)-\arctan\left(\frac{\beta\sin(\alpha)}{1-\beta\cos(\alpha)}\right)\right).$$
Setting $\varphi\coloneqq\theta+\alpha$ and $A(\alpha)\coloneqq\arctan\left(\frac{\beta\sin(\alpha)}{1-\beta\cos(\alpha)}\right)$ we have that 
$$A(\alpha)+\pi\xi+\frac{\alpha}{2}=A(\varphi)+\frac{\varphi}{2}$$
so that using the trigonometric identity
$$\tan(\arctan(x)+\arctan(y))=\frac{x+y}{1-xy},\qquad\text{if}\;\; xy\neq1$$
we obtain 
$$\tan\left(A(\alpha)+\pi\xi+\frac{\alpha}{2}\right)=\frac{A(\varphi)+\tan\left(\frac{\varphi}{2}\right)}{1-A(\varphi)\tan\left(\frac{\varphi}{2}\right)}.$$
Before proceeding with the main computation we have to rewrite $A(\varphi)$ in terms of $\tan\left(\frac{\varphi}{2}\right)$, so by using these trigonometric identities
$$\sin(\varphi)=\frac{2\tan\left(\frac{\varphi}{2}\right)}{1+\tan\left(\frac{\varphi}{2}\right)^2};\qquad\cos(\varphi)=\frac{1-\tan\left(\frac{\varphi}{2}\right)^2}{1+\tan\left(\frac{\varphi}{2}\right)^2}$$
we obtain that 
$$A(\varphi)=\frac{2\beta\tan\left(\frac{\varphi}{2}\right)}{(1-\beta)+(1+\beta)\tan\left(\frac{\varphi}{2}\right)^2}.$$
Putting everything together one obtains
$$\tan\left(\pi\xi+A(\alpha)+\frac{\alpha}{2}\right)=\frac{1+\beta}{1-\beta}\tan\left(\frac{\varphi}{2}\right)$$
so that we finally obtain
$$\frac{\varphi}{2}=\arctan\left(\frac{1-\beta}{1+\beta}\tan\left(\pi\xi+A(\alpha)+\frac{\alpha}{2}\right)\right)+k\pi$$
for some $k\in\Z$. From this and applying the previous argument even to $A(\alpha)+\frac{\alpha}{2}$, we conclude that 
$$F^{-1}_{\alpha,\beta}(\xi)=2\arctan\left(\frac{1-\beta}{1+\beta}\tan\left(\pi\xi+\arctan\left(\frac{1+\beta}{1-\beta}\tan\left(\frac{\alpha}{2}\right)\right)\right)\right)+2k\pi-\alpha.$$
The only missing detail is to determine the right $k\in\Z$. Since $F_{\alpha,\beta}:[0,1]\rightarrow[0,2\pi]$ is monotone and continuous the we expect the same from $F^{-1}_{\alpha,\beta}:[0,2\pi]\rightarrow[0,1]$. It is easy to see that the only point of discontinuity of $F^{-1}_{\alpha,\beta}$ is 
$$\xi=c\equiv\frac{1}{2}-\frac{1}{\pi}\arctan\left(\frac{1+\beta}{1-\beta}\tan\left(\frac{\alpha}{2}\right)\right)$$ with $F^{-1}_{\alpha,\beta}(c^+)-F^{-1}_{\alpha,\beta}(c^-)=2\pi$ so that exists $\Tilde{k}\in\Z$ such that
$$F^{-1}_{\alpha,\beta}(\xi)=\begin{cases}
        G_{\alpha,\beta}(\xi)+2\pi\Tilde{k}\qquad\qquad\; \text{if}\;\xi\in[0,c]\\
        G_{\alpha,\beta}(\xi)+2\pi(\Tilde{k}+1)\qquad \text{if}\;\xi\in(c,1]
    \end{cases}$$
and by imposing that $F^{-1}_{\alpha,\beta}(0)=0$, or equivalently $F^{-1}_{\alpha,\beta}(2\pi)=1$, we obtain that $\Tilde{k}=0$, as we wanted.

\end{document}